\documentclass[11pt]{amsart}
\usepackage{amssymb,amsmath,amsthm,amscd}
\usepackage[all]{xy}
\usepackage{color}
\numberwithin{equation}{section}

\newtheoremstyle{fancy1}{10pt}{10pt}{\itshape}{12pt}{\textsc\bgroup}{.\egroup}{8pt}{
}
\newtheoremstyle{fancy2}{10pt}{10pt}{}{12pt}{\itshape}{.}{8pt}{ }

\theoremstyle{fancy1}

\newtheorem{lem}[equation]{Lemma}

\newtheorem*{main*}{Theorem}

\newtheorem*{cor*}{Corollary}

\theoremstyle{fancy2}

\newtheorem*{rem*}{Remark}

\newcommand{\cref}[1]{Corollary~\ref{#1}}

\newcommand{\lref}[1]{Lemma~\ref{#1}}

\newcommand{\Sph}{\mathbb{S}}

\newcommand{\gS}{\mathsf{S}}
\newcommand{\gC}{\mathsf{C}}
\newcommand{\gH}{\mathsf{H}}
\newcommand{\gK}{\mathsf{K}}

\DeclareMathOperator{\diag}{diag}

\DeclareMathOperator{\SO}{SO}
\DeclareMathOperator{\Spin}{Spin}

\DeclareMathOperator{\ad}{ad}

\newcommand{\U}{\mathsf{U}}
\newcommand{\Lk}{\mathfrak{k}}

\newcommand{\su}{\mathfrak{su}}

\newcommand{\so}{\mathfrak{so}}
\newcommand{\symp}{\mathfrak{sp}}

\newcommand{\Sp}{\mathsf{Sp}}

\newcommand{\SU}{\mathsf{SU}}

\DeclareMathOperator{\Fix}{Fix}

\DeclareMathOperator{\Ad}{Ad}

\DeclareMathOperator{\pr}{pr}

\DeclareMathOperator{\rank}{rank}

\newcommand{\ml}{\langle}                     
\newcommand{\mr}{\rangle}                     

\newcommand{\Lp}{\mathfrak{p}}
\newcommand{\Lq}{\mathfrak{q}}

\newcommand{\RP}{\mathbb{R\mkern1mu P}}
\newcommand{\CP}{\mathbb{C\mkern1mu P}}
\newcommand{\HP}{\mathbb{H\mkern1mu P}}
\newcommand{\CaP}{\mathrm{Ca}\mathbb{\mkern1mu P}^2}

\newcommand{\C}{{\mathbb{C}}}
\newcommand{\R}{{\mathbb{R}}}
\newcommand{\Z}{{\mathbb{Z}}}

\newcommand{\Lh}{{\mathfrak{h}}}
\newcommand{\gL}{{\mathsf{L}}}
\newcommand{\gU}{{\mathsf{U}}}

\DeclareMathOperator{\Fi}{Fix(\iota)}
\DeclareMathOperator{\Fis}{Fix(\iota_2)}
\DeclareMathOperator{\id}{id}
\DeclareMathOperator{\Hi}{\gH^\iota}

\newcommand{\gE}{\mathsf{E} }
\newcommand{\gF}{\mathsf{F} }
\newcommand{\gT}{\mathsf{T} }
\newcommand{\gG}{\mathsf{G} }

\begin{document}
\date{\today}

\title{Revisiting homogeneous spaces with positive curvature}

\author{Burkhard Wilking}
\address{University of M\"{u}nster\\
       Einsteinstrasse 62\\
       48149 M\"{u}nster, Germany}
\email{wilking@math.uni-muenster.de}

\author{Wolfgang Ziller}
\address{University of Pennsylvania\\
       Philadelphia, PA 19104}
\email{wziller@math.upenn.edu}
\thanks{ The authors were supported by grants from the
DFG and NSF}

\maketitle


The classification of compact simply connected homogeneous spaces of positive curvature is now almost 40 years old. It has been accomplished in a series of papers by M.~Berger, N.~Wallach, Aloff-Wallach,
 and L.~B\'erard Bergery \cite{Be,Wa,AW,BB}, with an omission in \cite{Be} as observed in \cite{Wi1}.
   As was recently observed by J. Wolf and M. Xu \cite{WX}, there is a gap in
B\'erard Bergery's classification of odd dimensional positively curved  homogeneous spaces in
the case of the Stiefel manifold $\Sp(2)/\U(1)=\SO(5)/\SO(2)$.
 Since this classification has been used in several other  papers, for example, in the classification of
 positively curved cohomogeneity one manifolds in \cite{GWZ} and
  positively curved polar manifolds in \cite{FGT}, it seems desirable to correct this situation.
  We thus present here a modern complete and self contained proof of the
   classification, confirming that there are indeed no new examples.
   To be more precise we will reprove the following

   \begin{main*} Suppose a compact connected Lie group $\bar{\gK}$ acts isometrically, effectively and
   transitively on a simply connected manifold of positive sectional curvature
   with stabilizer group $\bar{\gH}$. Then the pair  $(\bar{\gK},\bar{\gH})$ is isomorphic to $(\gK/\gC,\gH/\gC)$
   for one of the triples $(\gK,\gH,\gC)$  in
   the Tables A or B.
   \end{main*}

\renewcommand{\arraystretch}{1.2}

\begin{table}[!h]
      \begin{center}
          \begin{tabular}{|c|c|c|c|c|}
\hline $\gK$          &$\gH$  &$\gK/\gH$& Kernel $\gC$& $N(\gH)/\gH$\\
\hline \hline

 $\SU(5)$ &$\Sp(2)\cdot\gS^1$ &$B^{13}$&$\Z_5$& $\{e\}$  \\
\hline
 $\Sp(2)$&$\Sp(1)_{max}$& $B^7$&$\Z_2$&$\{e\}$\\
\hline
 $\SU(3)\times \SO(3)$&$\U(2)$& $W^7_{1,1}$&$\Z_3$&$\{e\}$\\
\hline

$\SU(3)$ &$\gT^2$ & $W^6$&$\Z_3$& $\gS_3$ \\
\hline

 $\Sp(3)$ &$\Sp(1)^3$ & $W^{12}$&$\Z_2$& $\gS_3$ \\
\hline

 $\gF_4$ &$\Spin(8)$ &$W^{24}$&$\{e\}$& $\gS_3$ \\
\hline

 $\SU(3)$ &$\gS^1=\diag(z^p,z^q,\bar{z}^{p+q})$& $W^7_{p,q} $& $\Z_3$ if $p\equiv q$ mod $3$& $\gS^1$ if $p\neq q$ \\
  & $p\ge q\ge 1$, $\gcd(p,q)=1$ & &\{e\} if $p\not \equiv q$ mod $3$& $\SO(3)$ if $p=q$\\ \hline

 $\U(3)$ & $\gT^2(p,q)$ & $W^7_{p,q}$& $\Z_{p+2q}$& $\gS^1$  \\
\hline
          \end{tabular}
      \end{center}
      \vspace{0.1cm}
      \begin{center}
      \begin{minipage}{11.5cm}
     {Table A: Homogeneous spaces $M^n=\gK/\gH$ with positive sectional
curvature,  which are  not
 diffeomorphic to rank one symmetric spaces.}
\end{minipage}
\end{center}\label{homog}
\end{table}

As far as the embeddings of $\gH$ in $\gK$ are concerned, $\Sp(2)\gS^1$ is the normalizer of $\Sp(2)\subset \SU(5)$
embedded by the four dimensional representation,
$\Sp(1)_{max}$ is the unique 3 dimensional maximal subgroup of $\Sp(2)$, and in the third example $\gU(2)$ is the normalizer of $\Delta \SU(2)\subset \SU(2)\times \SO(3)\subset \gK$.
The last example is just an $\gS^1$ extension of the previous one and they are
the Aloff Wallach spaces \cite{AW}, with $\gT(p,q):=\{\diag(z^p\zeta,z^q,\bar z^{p+q})\mid z,\zeta \in \gS^1\}$
for $p\ge q\ge 1$, $\gcd(p,q)=1$.

The first three  examples are the only ones in the list  which admit
a normal homogeneous metric of positive sectional curvature.
The first two of those were discovered by Berger \cite{Be}, while the third was added in \cite{Wi1}
and is diffeomorphic to the Aloff Wallach space $W_{1,1}$.
  The Wallach flag manifolds $W^6$, $W^{12}$ and $W^{24}$ are
the only even dimensional positively curved simply connected homogeneous spaces
apart from the rank one symmetric space, \cite{Wa}.
We will not reprove that these spaces have positive curvature (see e.g. \cite{Zi} for details).
The group $\gC$ is defined as the intersection
of $\gH$ with the center of $\gK$ and corresponds to the kernel of the action of $\gK$ on $\gK/\gH$.
The normalizer $N(\gH)$ of $\gH$ in $\gK$ can be determined by a standard computation
and we kept track  of the isomorphism type of
$N(\gH)/\gH$ in the last column, where $\gS_3$ stands for the permutation group with $6$ elements.

\begin{table}[!h]
      \begin{center}
          \begin{tabular}{|c|c|c|c|c|}
\hline $\gK$          &$\gH$ &$\gK/\gH$&Kernel $\gC$ &$N(\gH)/\gH$           \\
\hline \hline

$\SO(n+1)$ &$\SO(n)$ &$\Sph^{n}$ &$\{e\}$& $\Z_2$ (for $n\ge 2$)\\
\hline

$\SU(n+1)$ & $\SU(n)$& $\Sph^{2n+1}$& $\{e\}$ &$\gS^1$ (for $n\ge 2$)  \\
\hline

$\U(n+1)$ & $\U(n)$ & $\Sph^{2n+1}$ &$\{e\}$ & $\gS^1$ \\
\hline

$\Sp(n+1)$ & $\Sp(n)$ &  $\Sph^{4n+3}$& $\{e\}$ & $\gS^3$ \\
\hline

$\Sp(n+1)\Sp(1)$ & $\Sp(n)\Delta\Sp(1)$& $\Sph^{4n+3}$&$\Delta \Z_2$& $\Z_2$
\\
\hline

$\Sp(n+1)\U(1)$ & $\Sp(n)\Delta\U(1)$&$\Sph^{4n+3}$& $\Delta \Z_2$&$\gS^1$
 \\
\hline

$\Spin(9)$   &$\Spin(7)$& $\Sph^{15}$& $\{e\}$ & $\Z_2$  \\
\hline

$\Spin(7)$   &$\gG_2$ &$\Sph^7$  & $\{e\}$ & $\Z_2$ \\
\hline

$\gG_2$   &$\SU(3)$  &$\Sph^6$  &$\{e\}$  & $\Z_2$\\
\hline
\hline

$\SU(n+1)$ &$\U(n)$  & $\CP^n$ & $\Z_{n+1}$& $\{e\}$ (for $n\ge 2$)\\
\hline
$\Sp(n+1)$ &$\Sp(n)\Sp(1)$ &$\HP^n$ &$\Z_2$& $\{e\}$ (for $n\ge 2$) \\
\hline

$\Sp(n+1)$ &$\Sp(n)\U(1)$  & $\CP^{2n+1}$&$\Z_2$& $\Z_2$ \\
\hline

$\gF_4$ &$\Spin(9)$ & $\CaP$ &$\{e\}$& \{e\} \\
\hline

\hline
          \end{tabular}
      \end{center}
      \vspace{0.1cm}
      \begin{center}
      \begin{minipage}{10cm}
     {Table B: Transitive actions on rank 1 symmetric spaces.}
\end{minipage}
\end{center}

\end{table}

Here we should add that $\Spin(8)$ acts in three different ways
transitively on $\Sph^7$, but up to outer automorphisms of $\Spin(8)$
the actions are isomorphic.

 A non simply connected homogeneous space arises from a simply connected space
$\gK/\gH$ by replacing $\gH$ by an extension $\hat{\gH}\subset N(\gH)$.
Thus the following corollary can be viewed as the classification
 of non simply connected homogeneous spaces of positive sectional curvature.

\begin{cor*} Let $(\gK,\gH,N(\gH)/\gH)$ be one of the triples in Tables A or B.
If $\gK/\gH$ is not a Wallach flag manifold, then
all finite subgroups of $N(\gH)/\gH$ give rise to a finite
extension $\hat{\gH}$ of $\gH$ such that some positively curved metric descends to
a $\gK$ invariant metric of
$\gK/{\hat\gH}$.
If $\gK/\gH$ is a Wallach flag manifold, then only the
$\Z_2$-extensions of $\gH$ allow for
some positively curved metric to descend.
\end{cor*}

Using the description in \cite{Wi1} of the  $\SU(3)$-equivariant principal bundle
$\SO(3)\rightarrow W_{1,1}^7\rightarrow \CP^2$.
Shankar \cite{Sh1} was the first to observe
that any finite subgroup $\gF\subset \SO(3)$ can be realized
as fundamental group of a positively curved  homogeneous space $W_{1,1}^7/\gF$.

Similarly one can use the $\Sp(n+1)$
equivariant $\gS^3$-principal bundle $\gS^3 \rightarrow \Sph^{4n+3}\rightarrow \HP^n$
to realize any finite subgroup $\gF\subset \gS^3$ as fundamental groups of a homogeneous space form
$\Sph^{4n+3}/\gF$.

By the corollary it is clear that
any positively curved homogeneous spaces with a noncyclic fundamental group
must be equivariantly diffeomorphic to one of these.
We should mention that the full isometry group of the examples in Table A  was determined in \cite{Sh2}
and various fundamental groups of {\em locally} homogeneous quotients
have been exhibited \cite{GS}  although  a classification is
open.

Except for the Wallach flag manifolds, all the examples have  positively curved metrics
which are $\Ad_{N(\gH)}$-invariant and thus the Corollary is immediate for these examples.
In the remaining cases one just has to use in addition that the fundamental group
of an even dimensional positively curved manifold has at most two elements by the Synge Lemma.
The three $\Z_2$ subgroups of $\gS_3$ are conjugate and
thus up to conjugation there is only one $\Z_2$ extension of $\gH$ for the Wallach flag manifolds.

The rest of the paper is devoted to the proof of the theorem.
The even dimensional case is treated in section~\ref{sec:even}
and the odd dimensional case in sections~\ref{sec:nonsemi}, \ref{sec:semi} and \ref{sec:simple}.
We explain the strategy in more detail at the end of the following section.

\smallskip

\section{Obstructions to positive curvature}
We will classify compact simply connected Riemannian homogeneous spaces
 $\gK/\gH$ with positive sectional  curvature.
 We can thus assume that  $\gK$ and $\gH$ are compact and connected, and that the
 normal subgroup common to both is at most finite.
We fix a biinvariant metric $Q$ on the Lie algebra $\Lk$ of $\gK$
and let $\Lp$ denote the orthogonal complement of the subalgebra $\Lh$ in $\Lk$.
The $\gK$-invariant metrics of $\gK/\gH$ are in one to one correspondence
with positive definite selfadjoint endomorphisms
$G\colon \Lp\rightarrow \Lp$  which commute with $\Ad_{\gH}$.
In fact, one can use $G$ to define
 a scalar product on $\Lp$ by putting $\ml x,y\mr=Q(Gx,y)$ for $x,y\in \Lp$ and
 extend this equivariantly to $\gK/\gH$.
We will implicitly assume that some $G$ has been chosen.
The following criteria for finding planes with $0$ or non-positive curvature is used as an
obstruction in our classification.

\begin{lem}\label{ev}
\begin{enumerate}
 \item[a)]
If $x,y\in \Lp$ are linearly independent eigenvectors of $G$ with $[x,y]=0$, then
they generate a zero curvature plane in $\gK/\gH$.
 \item[b)] Let $x\in \Lp$ be an eigenvector to the smallest eigenvalue $\lambda$ of $G$
 and assume we can find a linearly independent vector $z\in \Lp$ with $[x,z]=0$. If we put $y=G^{-1}z$, then $x,y$ generate a plane of non-positive curvature
 in $\gK/\gH$.
\end{enumerate}
\end{lem}

\begin{proof}
One can express the formula for the sectional curvature of the homogeneous metric in terms of the biinvariant metric (see e.g. \cite{Pu} or \cite{GZ}) as follows:
\begin{align*}
    \ml R(x,y)y,x\mr &= Q(B_-(x,y),[x,y])
   -\tfrac{3}{4}Q(G[x,y]_{\Lp},[x,y]_{\Lp}) \\
&\quad + Q(B_+(x,y),G^{-1}B_+(x,y)) - Q(B_+(x,x),G^{-1}B_+(y,y))
\end{align*}
where $B_{\pm} (x,y) = \tfrac{1}{2} \bigl([x,Gy] \mp [Gx,y]\bigr)$ and $[x,y]_{\Lp}$ is the $Q$-orthogonal projection
of $[x,y]$ to $\Lp$.  This clearly implies (a).

For part (b),  observe that $B_+(x,y)\in\Lp$ for all $x,y\in\Lp$.
 Indeed, since $\ad_v$ commutes with $G$ for $v\in \Lh$, this well known fact (see e.g. \cite{GZ},p.624  or \cite{BB},p. 62) follows from
 \[
 Q([x,Gy],v)=-Q( x, \ad_v Gy)=-Q(x,G\ad_v y)=-Q(Gx,[v, y])=Q([Gx,y],v).
 \]

 If $x,y,z$ are as specified, then $[x,Gy]=0$ and hence $B_-(x,y)=\tfrac12 \lambda [x,y]$, $B_+(x,y)=-\tfrac12 \lambda [x,y]$ and $B_+(x,x)=0$.
 Thus we also have $[x,y]\in\Lp$, and $x,y$ are linearly independent since $Gx=\lambda x$ and $Gy=z$ are. Altogether
\begin{align*}
\ml R(x,y)y,x\mr&= \tfrac12 \lambda Q([x,y],[x,y])+\tfrac 14 \lambda^2 Q([x,y],G^{-1}[x,y]) -\tfrac{3}{4} Q(G[x,y],[x,y])\\
&\le (\tfrac12+\tfrac14-\tfrac34)\lambda \bigl\|[x,y]\bigr\|_Q^2=0
\end{align*}
where we used the inequalities $Q(Gu,u)\ge \lambda \|u\|^2_Q$ and
$Q(G^{-1}u,u)\le 1/\lambda \|u\|^2_Q$ for all $u\in\Lp$.
\end{proof}

The lemma can also be used to give a new proof for the following essential  obstruction

\begin{lem}[Berger] If
$\gK/\gH$ is a positively curved $n$-dimensional homogeneous space,
 then $\rank \gK=\rank \gH$ if $n$ is even, and $\rank \gK=\rank \gH+1$ if $n$ is odd.
\end{lem}

\begin{proof} Consider first the special case of a trivial group $\gH$.
Then $\gK$ is endowed with a left invariant positively curved metric.
In our above notation, if $x\in \Lp=\Lk$ is an eigenvector to the minimal eigenvalue of $G$,
then by Lemma~\ref{ev} b) every vector  $z\in \Lk$ that commutes with $x$ is linearly dependent to $x$.
Thus $\rank(\gK)\le 1$.

We now use the following well known fact, which will also be a crucial tool for us later on.
Let $\gL\subset \gH\subset\gK$ and $C(\gL)_0$ the identity component of the centralizer of $\gL$ in $\gK$.
Then $C(\gL)_0$ acts transitively on the component of the fixed point set $\Fix(L)_0\subset \gK/\gH$ through $e\gH$,
as one easily sees by computing the tangent space of the orbit $C(\gL)_0\cdot e\gH$.
Thus $\Fix(\gL)_0=C(\gL)_0/C(\gL)_0\cap\gH$ is a totally geodesic submanifold of $\gK/\gH$ and hence
has positive curvature.

We now apply this to a maximal torus  $\gT\subset \gH$. Then $C(\gT)_0\cap\gH=\gT$ and hence $C(\gT)_0/\gT$ acts transitively and  freely on $\Fix(\gT)_0$.
 Thus $C(\gT)_0/\gT$ admits a positively curved left invariant metric and
by the above special case $\rank(C(\gT)_0/\gT)\le 1$ or equivalently
$\rank(\gK)\le \rank(\gH)+1$.
Since $\dim(\gL)-\rank(\gL)$ is an even number for any compact Lie group $\gL$,
the lemma follows.
\end{proof}

By \cite{Wi2}, if a group $\gK$  of the form $\SO(n),\SU(n),\Sp(n)$ acts isometrically on a positively curved manifold
in such a way that the principal isotropy group contains a $k\times k$ block of $\gK$ with $k\ge 3$,
then the underlying manifold is covered by a manifold which is homotopy equivalent to a rank one symmetric space.
In the homogeneous case one can strengthen it as follows, which will be our main tool in the classification since it allows one to proceed by induction on the dimension
of the Lie group. Although this result follows from \cite{Wi2} and the classification of homogeneous spaces homotopy equivalent to a rank one symmetric space, we give here a simple proof in the homogeneous case.

\begin{lem}[Block Lemma]\label{lem: block}
 Let $\gK\in \{ \SO(n),\SU(n),\Sp(n)\}$ and assume that a connected proper subgroup
 $\gH$ contains  the lower $k\times k$-block of $\gK$ with $k\ge 3$ if $\gK\in \{\SO(n),\SU(n)\}$ and $k\ge 2$ if
$\gK=\Sp(n)$. If $\gK/\gH$ admits an $\gK$-invariant positively curved metric, then
$\gH$ contains  a group conjugate to  a lower $(n-1)\times (n-1)$-block
and $(\gK,\gH)$ is one of the pairs listed in Table B.
\end{lem}

\begin{proof} We may assume that $k$ is chosen maximal among all groups which are conjugate to $\gH$.
We then let $\gL_k$ denote the lower $k\times k$-block,
$N(\gL_k)$ its normalizer and let
$\Lq\subset \Lk$ denote the orthogonal complement of $N(\gL_k)$.
An element $v\in \Lq$ is determined by its entries in the upper right $(n-k)\times k$ corner
and we say $v$  is in the $i$-th row if the entries  in the other rows are zero, $i=1,\ldots,n-k$.
 Under the action of $\gL_k$ the space $\Lq$ is  decomposed by the rows  into $(n-k)$ pairwise equivalent  sub-representations
 of $\gL_k$. If $\gK=\SO(n)$ or $\gK=\SU(n)$ with $k\ge 3$,
 these irreducible representations are orthogonal resp. unitary, and if $\gK=\Sp(n)$
  with $k\ge 2$ they are quaternionic.
  This in turn implies that the action of the upper $(n-k)\times (n-k)$-block $\gU_{n-k}$ on $\Lq$
  induces a transitive action on the irreducible $\gL_k$-sub-representations.

 We claim that $\gL_k$ is normal in $\gH$.
Otherwise, the isotropy representation of $\gH/\gL_k$ contains a nontrivial irreducible sub-representation of
$\gL_k$ which can be seen as a  sub-representation of $\Lq$.
As explained, we may assume that it is given by the last row in $\Lq$.
But now it is easy to see that $\gH$ contains the lower $(k+1)\times (k+1)$-block
in contradiction to our choice of $k$.

 Thus $\gL_k$ is normal in $\gH$.
 If $k=n-1$, then clearly $(\gK,\gH)$ is one of the pairs listed in Table~B.
  Suppose on the  contrary $k\le n-2$.  We
  choose an irreducible sub-representation $V\subset \Lq$
  consisting of eigenvectors of the selfadjoint endomorphism $G$.
  As before $\Ad_g V$ is given by the last row in $\Lq$ for some $g\in \gU_{n-k}$.
  After conjugating $G$ and $\gH$ with $\Ad_g$ we can therefore assume that
  the last row in $\Lq$ consists of eigenvectors.
  Clearly we can iterate this argument and we may assume without loss
  of generality that each row of $\Lq$ consists of eigenvectors of $G$.
 But now it is obvious that we can find commuting eigenvectors and by Lemma~\ref{ev} a) this is a contradiction.
\end{proof}

\begin{rem*}\label{gen}
The subgroups $\gU(2)\subset \SO(4)$ and  $\Sp(2)\subset \SU(4)$ contain a
$2\times 2$  but no $3\times 3$-block. Thus
we can not allow $k=2$ for $\gK\in \{\SU(n),\SO(n)\}$.
The proof breaks down, since for a $2\times 2$-block $\gL_2\subset \gK$
the irreducible sub-representation in $\Lq$ are complex if $\gK=\SO(n)$
and quaternionic for $\gK=\SU(n)$ and the upper $(n-2)\times (n-2)$-block
no longer acts transitively on them.
\end{rem*}

The proof of the theorem in the next few sections will go by induction on the dimension of the Lie group, that is,
at all times we will assume that the main theorem holds for all Lie groups with dimension strictly below $\dim(\gK)$.
 For any element $\iota\in \gH$ the fixed point set $\Fi$ is totally geodesic and hence positively curved.
 If $\Fi_0$ is the component of $\Fi$ containing the base point, then the id component of the  centralizer $C(\iota)_0$ acts transitively on it,
  with stabilizer group  $\Hi:=C(\iota)_0\cap \gH$ and thus $C(\iota)_0/\Hi$ is positively curved.
   For simplicity of notation we let $\Fi$ stand for $\Fi_0$   and $C(\iota)$ for $C(\iota)_0$.
    Since $\iota$ is contained in a maximal torus of $\gH$ which in turn can be extended to a maximal torus of
 $\gK$, we have $\rank C(\iota)=\rank \gK$, $\rank(\Hi)=\rank(\gH)$ and $\iota$ is a central element in the identity component of $\Hi$.
   Hence the codimension of these fixed point sets $C(\iota)/\gH^\iota$ is always even,
   and  we can do the induction in even and odd dimensions separately.

   In   all of the cases we will consider, $\Ad_\iota$ is an involution, and hence
  $\gK/C(\iota)$ is a symmetric space with $\rank(C(\iota))=\rank(\gK)$.
  If $\gK$ is an exceptional Lie group, then the classification of symmetric spaces
 only allows  the following possibilities for the pair $(\gK, \ C(\iota))$
\begin{align*}
&\gG_2 \ : \ \SO(4),\qquad \gF_4\ : \  (\Sp(3)\times\Sp(1))/\Delta\Z_2,\quad  \Spin(9) \\
 &\gE_6\ : \  (\Spin(10)\times\gS^1)/\Delta\Z_4 ,\quad (\SU(6)\times\SU(2))/\Delta\Z_2 \\
 &\gE_7 \ : \  (\Spin(12)\times\SU(2))/\Delta\Z_2 ,\quad  (\gE_6\times\gS^1)/\Delta\Z_3,\quad \SU(8)/\Z_2\\
    &\gE_8 \ : \  (\gE_7\times\SU(2))/\Delta\Z_2, \quad \Spin(16)/\Z_2=:\SO'(16)
\end{align*}

  We can assume,  by making the action ineffective if necessary, that the semisimple part of $\gK$ is simply connected.
But when $\gK=\Spin(n)$, we will usually replace it by $\SO(n)$,
at the expense of possibly making $\gK/\gH$ not simply connected.
We will use the following notational conventions

\begin{enumerate}
 \item[$\bullet$] If we write a Lie group $\gL$ as $\gL=\gL_1\cdot \gL_2$
 then $\gL_1$ and $\gL_2$ are normal subgroups generating $\gL$ and $\gL_1\cap \gL_2$ is
 finite.
 \item[$\bullet$] If a group is identified as $\SO(4)$ then $\SU(2)_\pm$ will stand for the
 two simple normal subgroups.
 \item[$\bullet$]  $G\colon \Lp\rightarrow \Lp=\Lh^{\perp}$ will always
 stand for a positive definite selfadjoint endomorphism
 inducing a positively curved metric on the homogeneous space $\gK/\gH$, see first paragraph of this section.
 Both $\gH$ and $\gK$ are connected and the kernel of the action of $\gK$ on $\gK/\gH$ is  finite.
\item[$\bullet$]
  Sometimes we will just say that we found commuting eigenvectors
 of $G$. This is same as saying that the case under consideration can not occur,
 as it contradicts Lemma~\ref{ev} a).
Similarly, if a situation arises
 where our induction hypothesis can be used
 to show that some fixed point component cannot have positive sectional curvature,
 it should be understood that we can move on to the next case.
\end{enumerate}

\section{The even dimensional case}\label{sec:even}

Since $\rank \gK=\rank \gH$, all irreducible sub-representations of $\gH$ in $\Lp$ are inequivalent,
and hence the metric is diagonal,
that is, every vector in an irreducible sub-representation of $\gH$ is an eigenvector of $G$.
 In addition, we can assume that $\gK$ is simple since otherwise $\gK/\gH$ is a product
 homogeneous space and every $\gK$ invariant metric is a product metric, contradicting positive curvature.
 Furthermore, all elements in $\gK$ are, up to conjugacy, also contained in $\gH$ and hence
 for each symmetric pair $(\gK,\gL)$ with $\rank(\gK)=\rank(\gL)$
 we can find an
 element $\iota$
in $\gH$ such that $C(\iota)\cong \gL$ and  $\Ad_\iota$ is an involution.
In case of the classical Lie groups we always assume that $\gH$ contains the classical
diagonal torus  and then $\iota\in \gH$ will be diagonal as well.

Since $C(\iota)/\Hi$
is again  an even dimensional space of
 positive curvature, only one factor in $C(\iota)$ can act effectively on $\Fi$,
  and the others must lie in $\Hi$ and hence in $\gH$.
  Furthermore, if $C(\iota)/\Hi$ is not a point, then the action of the remaining factor
  is listed in Table A  or B, since we always
  assume that the main theorem holds for Lie groups of dimension strictly smaller
  than $\dim(\gK)$.  We point out though that  low dimensional isomorphisms of Lie algebras
  sometimes give rise to less obvious presentations,  e.g. $ \CP^3=\SO(6)/\U(3)=\SO(5)/\U(2)$.

We now discuss each simple Lie group separately. Due to the low dimensional isomorphisms
$\Spin(5)\cong\Sp(2) $ and $\Spin(6)\cong\SU(4)$ we only need to consider $\Spin(n)$ for
$n\ge 7$, and in that case replace it by $\SO(n)$ for simplicity.
The aim is to confirm that $(\gK,\gH)$ is listed in Table A or B.

\subsection{$\gK=\SU(3)$} The only  rank two subgroups of $\SU(3)$ are $\gU(2)$ and $\gT^2$, and
both pairs correspond to listed examples.

\subsection{$\gK=\SU(4)$} In this case, the fixed point set $\Fi=\gS(\U(2)\U(2))/\Hi$ has positive curvature,
for $\iota=\diag(-1,-1,1,1)$,
 and thus $\Hi$, and hence also $\gH$, contains an $\SU(2)$-block, say the upper $2\times 2$-block.

We next look at the involution $\iota_2=\diag(1,-1,-1,1)$ and repeat the argument. It shows that
$\gH$ also contains either the middle $\SU(2)$-block or the $(1,4)$-$\SU(2)$-block.
Neither of these blocks commutes with the upper $\SU(2)$-block and in fact the two blocks generate
an $\SU(3)$-block and we are done by the Block Lemma.

\subsection{$\gK=\SU(k)$ ($k\ge 5$)}  In this case we look at the fixed point set
of $\iota=\diag(\zeta,\ldots,\zeta,-\zeta)\in \gH$ where $\zeta\in \gS^1$ is a primitive $2k$-th root of unity.
Thus $C(\iota)= \gU(k-1)$  and the fixed point component
$\U(k-1)/\Hi$ admits positive curvature. Since $k-1\ge 4$, it follows that $\gH$
contains a $(k-2)\ge 3$ block, and we are done by the Block Lemma.

 \subsection{$\gK=\Sp(2)$} We may assume $\iota=\diag(-1,1)\in \gH$.
 Since $\Sp(1)\times \Sp(1)/\gH^\iota$ admits positive curvature, $\gH$ contains an $\Sp(1)$ block.
 Thus either $\gH=  \Sp(1)\cdot \Sp(1)$ or $\gH= \Sp(1)\cdot\gS^1$ and
 both quotients are listed in Table B.

 \subsection{$\gK=\Sp(3)$}{ We may assume that the involutions $\iota_1=\diag(-1,1,1)$, $\iota_2=\diag(1,-1,1)$
 and $\iota_3=\diag(1,1,-1)$ are in $\gH$ and that $\gH$ contains no $2\times 2$-block.
 Then $\Fix(\iota_h)=\Sp(2)\Sp(1)/\gH^{\iota_h}$ and since $\gH$
 does not contain the $\Sp(2)$, it must contain the $\Sp(1)$-block --
 given by the $\gS^3$ in the $h$-th diagonal entry. Hence
 $\{\diag(a,b,c)\mid a,b,c\in \gS^3\} \subset \gH$. Since $\gH$ contains no $2\times 2$-block equality must hold and we
 are left with the 12 dimensional Wallach flag manifold.

 \subsection{$\gK=\Sp(k)$ ($k\ge 4$)} We may assume that  $\iota=\diag(-1,-1,1,\ldots,1)\in\gH$
 and hence $\Fi=\Sp(2)\cdot \Sp(k-2)/\gH^\iota$. Thus
 $\gH$ contains either the upper $2\times 2$ block or the lower $(k-2)\times (k-2)$-block.
 In either case the result follows from the Block Lemma.

 \subsection{$\gK=\SO(k)$ ($k\ge 7$)} We may assume $\iota=\diag(-1,-1,-1,-1,1\ldots,1)\in\gH$ and hence
 $\Fi=\SO(4)\cdot \SO(k-4)/\gH^\iota$ admits positive curvature. The group
 $\gH^\iota$ must contain all but one of the connected simple normal subgroups of $\SO(4)\cdot \SO(k-4)$.
 This implies that $\gH$ contains either the upper $4\times 4$-block or the lower $(k-4)\times (k-4)$-block
and we are done by the Block Lemma.

 \subsection{$\gK=\mathsf{G}_2$} For any involution $\iota\in \gH$, we have $\Fi=\SO(4)/\gH^\iota$, which implies that $\Hi$
  contains at least a group
 isomorphic to $\gU(2)\subset \SO(4)$. We claim that  $\gH=\gU(2)$  can not hold.
 In fact, otherwise  we could choose $\iota$ as a non-central involution in $\gU(2)$ and would get
 $\gH^\iota =\gT^2$, but $\SO(4)/\gT^2$ does not admit positive curvature. Thus $\gH$ is strictly bigger
 than $\gU(2)$. The only connected proper subgroups of $\mathsf{G}_2$
 satisfying this are $\SU(3)$ and $\SO(4)$.
 In the former case, $\mathsf{G}_2/\SU(3)\cong \Sph^6$ is listed in Table B
while the latter case is not possible as $\mathsf{G}_2/\gH$ would be isometric to the rank 2 symmetric space
 $\mathsf{G}_2/\SO(4)$.

 \subsection{$\gK=\gF_4$} Choose an involution $\iota\in \gH$ whose centralizer in $\gF_4$ is given by $\Spin(9)$.
 Since $\Fi=\Spin(9)/\gH^\iota$ has positive curvature,
  $\gH$ contains $\Spin(8)$. The only proper connected subgroups in $\gF_4$ satisfying this are
 $\Spin(8)$ and $\Spin(9)$ and and both correspond to listed quotients.

 \subsection{$\gK=\gE_i$, $i=6,7,8$}
 We choose an element $\iota\in\gH$
 whose centralizer  $C(\iota)$ is given  by $\gS^1\cdot \Spin(10)$  if $i=6$, $\gS^1\cdot \gE_6$ if $i=7$
 and   $\SO'(16)$ if $i=8$.
 By induction we can use the Tables A and B in the introduction to see that
 $C(\iota)$ can not act transitively by isometries on a positively curved manifold of positive even dimension.
 Hence $C(\iota)=\Hi\subset \gH$ and equality must hold
 since $C(\iota)$ is maximal -- a contradiction as
 $\gK/C(\iota)$ is a higher rank symmetric space.

\section{$\gK$ not semisimple}\label{sec:nonsemi}
In the remaining three sections we assume that $\gK/\gH$ is an almost effective representation
of an odd-dimensional
homogeneous space of positive sectional curvature and in this section
we treat the case of a non-semisimple compact group $\gK$.
Since $\rank(\gK)=\rank(\gH)+1$, the center of $\gK$ can be at most one dimensional.
After passing to a finite cover
we can assume $\gK=\gS^1\times \gK_2$ with $\gK_2$  semisimple.
We let $\pr_2(\gH)$ denote the projection of $\gH$ to the second factor.
Since $\gK/\gH$ has finite fundamental group,
the projection of $\gH$ to the first factor is surjective.

If we put $\gH_2=\gK_2\cap \gH$, then $\gH = \Delta\gS^1\cdot \gH_2$.
Since the projection to the first factor is surjective, $\gK_2$
acts transitively on $\gK/\gH$ with stabilizer $\gH_2$ and by induction on the dimension of the Lie group
$(\gK_2,\gH_2)$ is up to a finite covering one of the pairs listed in Table A or B.
The group $\pr_2(\gH)$ is contained in the normalizer of $\gH_2$ in $\gK_2$
and thus $N(\gH_2)/\gH_2$ is at least one dimensional. Combining this with the fact that $\gK_2$ is semisimple
we deduce that
$(\gK_2,\gH_2)$ is given by $(\SU(n),\SU(n-1))$, $(\Sp(n),\Sp(n-1))$
or $(\SU(3),\diag(z^p,z^q,\bar{z}^{p+q}))$ (with $p\ge q\ge 1$ and $\gcd(p,q)=1$).
In either case the corresponding $\gS^1$ extension is also listed in Tables A and B
and thus we are done.

\section{$\gK$  semisimple but not simple}\label{sec:semi}

We assume in this section $\gK=\gK_1\times \gK_2$ is  a simply connected product group
with semisimple factors of positive rank.
Notice that
$\rank(\gK_i)-\rank(\gH_i)\le \rank(\gK)-\rank(\gH)=1$ holds for
$\gH_i=\gK_i\cap \gH$, $i=1,2$.
We distinguish among three cases.

\subsection{$\gH_1$ and $\gH_2$ are finite.}\label{semi_1}
Then $\rank(\gK_i)\le \rank(\gH_i)+1=1$ and thus
$\gK=\gS^3\times \gS^3$. If $\gH$ is three dimensional then it is necessarily given
by $\Delta \gS^3$ and $\gK/\gH\cong \Sph^3$ is in our list.
Otherwise $\gH$ is a circle and we can assume $\gH=\{(z^p,z^q)\mid z\in \gS^1\}$
with $p\ge q\ge 1$
and $\gcd(p,q)=1$.
We want to rule out these potential examples by finding commuting eigenvectors.
The tangent space $\Lp$ splits into a trivial
 $1$-dimensional module and a $4$-dimensional module spanned by $(j,0),(k,0),(0,j),(0,k)$
 on which $\gH$ acts as a rotation
 on the span of
  the first two and the last two vectors.
If $p\neq q$, the
sub-representations are inequivalent and thus $G$-invariant.
Therefore  $(j,0)$ and $(0,j)$ are commuting eigenvectors.
If $p=q$ we can assume that  one eigenvector is given by
   $e_1=(\alpha j, \beta j e^{i\psi})$ for some $\alpha,\beta,\psi\in \R$ with $\alpha^2+\beta^2=1$.
    A second eigenvector to the same eigenvalue is then obtained by the action of $\gH$
    to be $e_2=(\alpha k, \beta k e^{i\psi})$.
Any vector in the four dimensional modul which is $Q$-orthogonal to both,
must thus be an eigenvector as well.
Hence $e_3=(-\beta j, \alpha j e^{i\psi})$ is an eigenvector and it clearly commutes with $e_1$.
Here we used indirectly that we can choose the biinvariant metric $Q$ such that both factors are weighted
equally.

\subsection{ $\gH_1$ is finite but $\gH_2$ is not.}\label{semi_2}
Then $\rank(\gK_1)=1$ and thus $\gK=\gS^3\times \gK_2$.

We start with the case where
 $\gH$ projects surjectively to the first factor.
 Then
$\gH= \Delta \gS^3\cdot \gH_2$
and the factor $\gK_2$ acts transitively
on the homogeneous space $\gK/\gH$ with stabilizer $\gH_2$.
By induction on the dimension of the Lie group, the pair $(\gK_2,\gH_2)$ is up
to a finite covering  listed in Tables A or B.
Since $\pr_2(\gH)$ is contained in the normalizer of $\gH_2$ it follows that $N(\gH_2)/\gH_2$
is three dimensional.
Therefore the pair $(\gK_2,\gH_2)$ is given by
$(\SU(3),\diag(z,z,\bar{z}^2))$ or $(\Sp(n),\Sp(n-1))$.
But then $(\gK,\gH)$ is (up to finite kernel) either $(\SU(3)\times \SO(3),\gU(2))$
or $(\Sp(n)\times \Sp(1),\Sp(n)\Delta \Sp(1))$ and both are listed in Tables A and B.

 The projection of $\gH$ to the $\gS^3$ factor can not be trivial as otherwise  $\gH=1\times \gH_2$
 would be a product subgroup of $\gS^3 \times \gK_2$ and every invariant metric of $\gK/\gH$ would
 be a product metric.

It remains to consider the subcase where the projection of $\gH$ to the $\gS^3$ factor is given by an $\gS^1$.
Let $N(\gH_2)$ denote the normalizer of $\gH_2$ in $\gK_2$.
If $N(\gH_2)/\gH_2$ is three dimensional then a fixed point component of $\gH_2$ is
locally isometric to $\bigl(\gS^3\times \gS^3\bigr)/\gS^1$, which is impossible as we saw in \ref{semi_1}.

Otherwise, we have $\dim(N(\gH_2)/\gH_2)\le 1$.
This implies that the two dimensional irreducible representation of $\gH$ in the first factor is not equivalent to any other sub-representation of $\gH $ in $\Lp$
since no other nontrivial sub-representation has $\gH_2$ in its kernel.
Thus the two dimensional sub-representation in the first factor consists of eigenvectors and it is now easy to
find commuting eigenvectors of $G$.

\subsection{ Both $\gH_1$ and $\gH_2$ are infinite.}
Then there is a nontrivial irreducible sub-representation
which does not contain $\gH_1$ in its kernel. Since  any such
sub-representation is tangent to the first factor there are eigenvectors in the first factor.
Similarly there are also eigenvectors
in the second factor and thus we found commuting eigenvectors.

\section{The odd dimensional case with $\gK$ Simple}\label{sec:simple}

The proof is again by induction on the dimension of the group.
Again we frequently use that the centralizer $C(\iota)$ of an element
$\iota \in \gH$ acts transitively on an odd-dimensional fixed point component of $\iota$
and by induction this action  is, up to possibly a larger kernel, (locally) given
by one listed in  Tables A and B in the introduction.
We point out though that low dimensional isomorphisms of Lie algebras
  sometimes give rise to less obvious presentations:
  \begin{align*}
 &\RP^7=\SO(5)/\SU(2)_\pm=\SO(6)/\SU(3)=\SO(7)/\gG_2=\SO(8)/\Spin(7)\\
 &\RP^7=\SO(5)\SU(2)/\SU(2)_-\cdot\Delta\SU(2)_+  \ ,\  \Sph^5=\SU(4)/\Sp(2), \quad \RP^{15}=\SO(9)/\Spin(7).
\end{align*}

As explained, for $\iota\in \gH$, the group $\gH$ is  either equal to $\Hi=C(\iota)\cap \gH$
or an equal rank enlargement thereof. The latter
 are rather rare, as follows
 from the table in \cite[p. 281]{Wo}.
  For example, up to covers the only equal rank enlargements of  simple Lie groups
  are $\SO(2n)\subset\SO(2n+1)$,  $\SU(3)\subset\gG_2$,
  $\Spin(9)\subset \gF_4$, $\SU(8)/\Z_2\subset \gE_7$,  $\SU(9)/\Z_3\subset \gE_8$ and $\Spin(16)/\Z_2\subset \gE_8$.
 The group $\gS^3\times \gS^3$ has only $\Sp(2)$ as equal rank enlargement whereas
  $\SO(4)$ has $\SO(5)$ and $\gG_2$.

We
 go through the list of simple Lie groups. By passing to a $\Z_2$ quotient if necessary we again can
deal with the group $\SO(k)$ rather than $\Spin(k)$
as long as we allow fundamental group $\Z_2$ for $\gK/\gH$.
As in the even dimensional case we only need to consider this for $k\ge 7$.

\subsection{$\gK=\SU(3)$}

If $\gH=\SO(3)$, then $\gK/\gH$ is isometric to a symmetric space of rank 2, and hence does not have positive curvature, and if $\gH=\SU(2)$, it is a sphere.
Otherwise, $\gH$ is one dimensional and we may assume
$\gH=\diag(z^p,z^q,\bar z^{p+q})$ with $p\ge q\ge 0$ and $\gcd(p,q)=1$.
Then $\gK/\gH$  is an Aloff Wallach space, which has positive curvature unless $(p,q)=(1,0)$.
 In the latter case we choose the involution $\iota=\diag(-1,1,-1)\in\gH$ whose fixed point set $\Fi=\U(2)/\gH=(\Sph^2\times\Sph^1)/\Delta\Z_2$ cannot have positive curvature.

\subsection{$\gK=\SU(4)$}\label{subsec:su4}

There exists an involution $\iota\in \gH$ which is not central since $\rank(\gH)=2$. In this case, $\Fi=\gS(\U(2)\U(2))/\Hi$, which can only have positive curvature
 if $\gH$ contains $\Delta\SU(2)$ or an $\SU(2)$-block. In the first case, $\gK/\gH$ is effectively a quotient of $\SO(6)$. The image of $\gS(\U(2)\U(2))$ in $\SO(6)$
  is $\SO(2)\SO(4)$, and hence the image of $ \Delta\SU(2)$ is a  $3\times 3$ block in $\SO(6)$, and we are done by the Block Lemma.

If $\gH$ contains a block (say lower) $\SU(2)$, there are  4 possible enlargements of rank 2.
If $\gH$ is simple,  it must be $\SU(3)$ or $\Sp(2)$, since $\SU(4)$ does
 not contain a $\gG_2$. But then $\gK/\gH=\Sph^7$ or $\Sph^5$ is in our list.
 A third possibility is that $\gH=\SU(2)\cdot \SU(2)$ (lower and upper $2\times 2$ block), but then
 $\gK/\gH$ is effectively given by $\SO(6)/\SO(4)$ and we are done  by the Block Lemma.

 The final possibility is $\gH=\SU(2)\cdot\diag(z^{2p},z^{2q},\bar z^{p+q},\bar z^{p+q})$
 for some $p,q\in \Z$   with
 $\gcd(p,q)=1$ and this can be ruled out as follows.
 If $|p|\neq |q|$,
 then the two rows in the orthogonal complement $\Lq$ of $\gS(\gU(2)\gU(2))$ correspond to inequivalent
  representations and hence are contained in eigenspaces of $G$ which clearly yields commuting
  eigenvectors.
  If $p=q$ the representation of $\gH$ in $\Lq$ decomposes into
  two equivalent {\em complex}
  representations and the normalizer of $\gH$ contains the upper $2\times 2$ block.
  Hence we can argue as in the proof of the Block Lemma to find commuting eigenvectors.
  If $p=-q$, then the involution $\diag(1,1,-1,-1)\in\gH$
  has a three dimensional fixed point component
  $\gS(\U(2)\U(2))/\gH=(\Sph^2\times\Sph^1)/\Z_2$
  with infinite fundamental group.

\subsection{$\gK=\SU(5)$}
 Let $\gF\cong \Z_2^4$ denote the group of diagonal matrices
 in $\SU(5)$ with eigenvalues  $\pm 1$.
 Since  we can assume that the  three dimensional torus in $\gH$ is diagonal, there exists an index $2$ subgroup $\gE$ of  $\gF$
 contained in $\gH$.
We claim that one element in $\gE$  has an eigenvalue $-1$ with multiplicity $4$.
Suppose not. If  $\iota_1,\iota_2\in \gF$ are two
elements  both of which
have the  eigenvalue $-1$
with multiplicity $4$, then it would follow $\iota_1\cdot \iota_2\in \gE$.
But these products generate the whole group $\gF$ -- a contradiction.

Therefore without loss of generality $\iota=\diag(-1,-1,-1,-1,1)\in \gH$.
By induction $\Fi=\U(4)/\Hi$,  must be one of
 $\U(4)/\U(3)$, $\gU(4)/\SU(4)$ or  $\SU(4)/\Sp(2)$. In the first two cases, $\gH$ contains a $3\times 3$ block,
 and we are done. In the last case,
$\gH$ contains $\Sp(2)\cdot\gS^1$, which gives rise to the positively curved Berger space
$\SU(5)/\Sp(2)\cdot\gS^1$.

\subsection{$\gK=\SU(k), \, k\ge 6$}
 We can assume that the maximal diagonal torus of $\gH$  has at least a one dimensional intersection with the maximal torus
of $\Delta \SU(3)\subset \SU(3)^2$ contained in the upper $6\times 6$ block of
$\SU(k)$.
 Clearly any involution $\iota$ in this intersection has a complex eigenvalue $-1$ with multiplicity $4$.
Then $\Fi=\gS(\U(4)\cdot \U(k-4))/\Hi$.

If $k\ge 7$, then $\Hi$ either contains the upper $4\times 4$ or the lower
$(k-4)\times (k-4)$ block and we are done.

If
$k=6$ then  the stabilizer group of the action of  $\SU(4)$ on $\Fi$ is
$\Sp(2)$ unless it contains an $\SU(3)$-block.
   Thus we may assume $\Sp(2)\SU(2)\cdot\gS^1\subset \gH$ with $\gS^1=\diag(z,z,z,z,\bar z^2,\bar z^2)$.
   This in turn implies that $\iota_2=\diag(i,-i,i,-i,i,-i)\in \Sp(2)\SU(2)\subset \gH$.
   The centralizer $C(\iota_2)$ is isomorphic to $\gS(\gU(3)\gU(3))$ and
   hence Fix$(\iota_2)=\gS(\gU(3)\gU(3))/\gH^{\iota_2}$ has positive curvature.
   By induction $\gH$ contains up to conjugation a $3\times 3$ block and we are done by the Block Lemma.

\subsection{$\gK=\Sp(2)$}\label{subsec:sp2}
There are  three 3-dimensional subgroups of $\Sp(2)$.
One quotient is a sphere $\Sph^7=\Sp(2)/\Sp(1)$, the second the Berger space
$\Sp(2)/\Sp(1)_{max}$ with positive curvature,
 and the third the Stiefel manifold $\Sp(2)/\Delta\Sp(1)=\SO(5)/\SO(3)$ which contains a $3\times3$ block,
and thus is ruled out by the Block Lemma.

It remains to consider
$\gH=\diag(z^p,z^q)$ with $p\ge q\ge 0$ and $\gcd(p,q)=1$. Then the weights of the adjoint action of $\gH$ on
$\Lp$ are $0, 2p,2q,p-q,p+q$. If they are all distinct, the metric $G$ is diagonal,
and there are two commuting eigenvectors.
If $(p,q)=(1,0)$, then $\Fi=\Sp(1)\Sp(1)/\gS^1\times 1$ cannot have positive curvature, where $\iota$ is the involution in $\gH$.

This leaves us with 2 exceptional cases. If $(p,q)=(1,1)$, then $\Sp(2)/\Delta \gS^1=\SO(5)/\SO(2)$
is the Stiefel manifold where we can think of $\SO(2)$ as the lower $2\times 2$ block.
In this case $\Lp=\Lp_0\oplus\Lp_1$ where $\gH$ acts trivially on the 3-dimensional module $\Lp_0$
(upper $3\times 3$-block),
 and as the direct sum of 3 equivalent 2-dimensional representations on $\Lp_1$.
 We use \lref{ev} (b) to find an obstruction. If an eigenvector corresponding to the smallest eigenvalue of $G$  lies in $\Lp_0$,
 then it has rank 2 in $\so(5)$ and  there is a vector in $\Lp_1$ that commutes with it. If the eigenvector lies in $\Lp_1$, we can use the fact that $\SO(5)/\SO(3)\SO(2)$
   is a symmetric space of rank 2 to  find a linearly independent vector in $\Lp_1$ that commutes with it.
   In either case \lref{ev} (b) implies that $\Sp(2)/\gH$ does not have positive curvature.

  Ruling out the
remaining case of $(p,q)=(3,1)$ we postpone to the end since it is
the only case that requires a more detailed argument, see section \ref{subsec:31}.

\subsection{$\gK=\Sp(3)$}\label{subsec:sp3}

If $\iota\in \gH$  is an involution which is not central, then $\Fi=\Sp(2)\Sp(1)/\Hi$ and by induction there are only 4 odd dimensional quotients which have positive curvature, corresponding to $\Hi=\Sp(2), \, \Sp(1)\Delta\Sp(1), \, \Sp(1)\Sp(1)$ or $\Sp(1)_{max} \Sp(1)$.
By the Block Lemma we may assume that $\gH$ does not contain a $2\times 2$ block and thus
$\gH=\Hi$.

If $\Hi$ is given by $\Sp(1)\Delta\Sp(1)=\diag(q,r,r)$ or $\Sp(1)\Sp(1)=\diag(1,q,r)$ with $q,r\in\Sp(1)$,
then we can  choose a second involution
$\iota_2=\diag(1,-1,-1)\in\gH$ with fixed point set $\Sp(1)\Sp(2)/\gH
=\Sp(2)/\Delta\Sp(1)$ in the first case, and $\Sp(1)\Sp(2)/\gH
=\Sp(1)\times(\Sp(2)/\Sp(1)\Sp(1))$ in the second case. Neither one admits positive curvature.

Thus we are left with the case $\Hi=\gH=\Sp(1)_{max} \Sp(1)\subset \Sp(2)\Sp(1)$.
Then $\Lp=\Lp_1\oplus \Lp_2$ where $\Lp_1$ is the irreducible sub-representation of $\gH$ given by
the orthogonal complement of $\Sp(1)_{max}$ in the upper $\symp(2)$-block and $\Lp_2$
is the irreducible inequivalent sub-representation given by $(\symp(2)\oplus \symp(1))^\perp$.
Clearly $\Lp_i$ is contained in an eigenspace of $G$.
We may assume that a maximal torus of $\gH$ is given by $\{\diag(z^3,z,\zeta)\mid z,\zeta\in \gS^1\}$ (see e.g. \cite{Be}, p.237).
 The circle $\diag(z^3,z,1)$ acts on the
Lie algebra of $\Sp(1)_{max}$ with weight 2 and thus $x=\diag(j,0,0)\in \Lp_1$, as $x$ lies in a $2$-dimensional sub-representation of
the circle with weight $6$.
Clearly we can find a commuting eigenvector in $\Lp_2$.

\subsection{$\gK=\Sp(k),\ k\ge 4$}

Fixed point groups of non central involutions are of the form $\Sp(r)\Sp(s)$ with $r+s=k$ and $r\ge s$.
Thus either $r\ge 3$ or $r=s=2$. In either case, it follows from our induction hypothesis that $\Hi\subset \gH$
contains a $2\times 2$ block and we are done.

\subsection{$\gK=\SO(k)$, $k\ge 7$}
We may assume that the maximal torus of $\gH$ has at least a one dimensional intersection with
$\SU(3)\subset \gU(3)$ contained in the upper the $6\times 6$-block.  Any involution
$\iota$ in this intersection has the eigenvalue $-1$ with  real multiplicity four. Hence without
loss of generality
$\iota=\diag(-1,-1,-1,-1,1,\ldots,1)\in \gH$.

A fixed point component is given by $\SO(4)\SO(k-4)/\Hi$.
 If $k\ge 10$, then by induction $\Hi$ either contains
the upper $4\times 4$-block or the lower $(k-4)\times (k-4)$ block and we are done.

If $k=7$, then $\Hi$ contains a $3\times 3$-block unless
 $\Hi_0=\SU(2)_-\cdot\Delta \SU(2)\cong \SO(4)$ with $\Delta \SU(2)\subset\SU(2)_+\cdot\SO(3)$.
If $\gH=\Hi$ then we consider another involution $\iota_2=\diag(-1,-1,1,1,-1,-1,1)\in\gH$ coming from $(\diag(i,-i),\diag(i,-i))\in \SU(2)_-\cdot\Delta \SU(2)$ with fixed point set $\SO(4)\SO(3)/\gH^{\iota_2}$ which cannot have positive curvature since $(\gH^{\iota_2})_0=\U(1)\U(1)$.
Thus $\gH\neq \Hi$ and $\gH$ is simple.
 Since $\SO(5)$ only embeds
 as a $5\times 5$ block into $\SO(7)$,
  it necessarily follows  that $\gH=\gG_2$ and $\SO(7)/\gG_2\cong \RP^7$
  is two fold covered by a listed example.

If $k=8$, then at least two of the four normal connected simple subgroups
of $\SO(4)\SO(4)$ are in $\Hi$. If they form a $4\times 4$ block
then we are done by the Block Lemma. Otherwise, we can choose a suitable
complex structure such that $\Hi$ contains the subgroup $\gL$ given as the upper and lower
$2\times 2$ block of $\SU(4)\subset \SO(8)$.
If we choose an automorphism of the Lie algebra $\so(8)$ that moves the subalgebra $\su(4)$
into a $6\times 6$ block $\so(6)$, then the image of the Lie algebra of $\gL$ will
be a $4\times 4$ block (see first paragraph of section~\ref{subsec:su4}).
By the Block Lemma  $\gH\cong \Spin(7)$ and $\SO(8)/\Spin(7)\cong \RP^7$ is two fold covered by a listed example.

It remains to consider $k=9$.
Then $\Fi=\SO(5)\SO(4)/\Hi$. By the Block Lemma, we can assume that both
$\SO(5)$ and one of the simple factors in $\SO(4)$ must act non trivially on $\Fi$.
This leaves only the possibility that effectively
$$\Fi=\RP^7=\SO(5)\SU(2)_+\bigm/\SU(2)_-\cdot\Delta\SU(2)_+$$
Thus $\Delta\SU(2)_+\subset\gH$ which implies that $\gH$ contains another involution $\iota_2$ with eigenvalue $-1$ of multiplicity 8 and hence has fixed point
component $\SO(8)/\gH^{\iota_2}$. Thus either $\gH$ contains a
$7\times 7$ block and we are done, or $\Fis=\RP^7=\SO(8)/\Spin(7)$. But then $\gH=\Spin(7)$ and
$\gK/\gH=\RP^{15}$ is two fold covered by a listed example.

\subsection{$\gK=\gG_2$}
For every involution $\iota\in \gH$, the fixed point set is $\Fi=\SO(4)/\gH^\iota$.
Since $\iota$ is contained in the center of $\Hi$ we deduce that $\Hi\not\cong \SO(3)$ and thus $\Hi=\gH$
is a normal subgroup of $\SO(4)$.
We let $\Lp_0\subset \Lp$ denote the $3$-dimensional trivial sub-representation of $\gH$
corresponding to the dual normal subgroup of $\SO(4)$.
The orthogonal complement $\Lq: = (\Lp_0)^{\perp}\cap \Lp$ corresponds to the
tangent space of the rank $2$ symmetric space $\gG_2/\SO(4)$.
Thus $G_{|\Lq}$ can not be a multiple of the identity
because otherwise we could find commuting eigenvectors.
This implies that the representation of $\gH\cong \SU(2)$ on the $8$-dimensional space $\Lq$
is reducible. Since $\Ad_{\iota|\Lq}=-\id$ there are no trivial or three dimensional sub-representations and it
 must be given as the sum of two $4$-dimensional sub-representations.

Altogether, $G_{|\Lq}$ has two  four dimensional eigenspaces
$\Lp_1$ and $\Lp_2$ corresponding to distinct eigenvalues $\lambda_1$ and $\lambda_2$.
If $X\in \Lp_1\setminus \{0\}$, then $\ad_X$ moves $\Lp_2$ to $\Lp_0\oplus \Lh$
as $\gG_2/\SO(4)$  is a symmetric space.
Since $\Lp_0$ is three dimensional, $[X,Y]\in \Lh$ for some $Y\in \Lp_2\setminus \{0\}$.
Now $B_+(X,Y)=[GX,Y]-[X,GY]=(\lambda_1-\lambda_2)[X,Y]\in\Lp$ (see proof of Lemma~\ref{ev})
gives $[X,Y]=0$ and we found two commuting eigenvectors.

\subsection{$\gK=\gF_4$}

In this case $C(\iota)$ is  either $\Sp(3)\cdot\Sp(1)$ or $\Spin(9)$. Since $\rank(\gH)=3$,
one of its involutions lies in $\Z_2\oplus\Z_2$, the center of $\Spin(8)\subset\Spin(9)$.
Since $\Sp(3)\Sp(1)$ does not contain
a $\Spin(8)$, the fixed point set is $\Fi=\Spin(9)/\Hi$.
Thus $\Hi=\Spin(7)$ since $\Hi$ has rank 3,
which in turn implies that $\gH=\Spin(7)$ since it has no equal rank enlargement.

Choose an involution $\iota_2\in \Spin(7)$ with $\gH^{\iota_2}=\Spin(4)\Spin(3)$.
Again $C(\iota_2)\cong \Sp(3)\cdot\Sp(1)$ or $\Spin(9)$ and by induction
$C(\iota_2)/\Hi$ does not have positive curvature.

\subsection{$\gK=\gE_6$}
Here
$C(\iota)$ is  either $\SU(6)\cdot\Sp(1)$ or $(\Spin(10)\times\gS^1)/\Delta\Z_4$.
As in the previous case, we can choose an involution
$\iota\in Z(\Spin(8))\cap\gH$ with $\Fi=\Spin(10)\cdot\gS^1/\Hi$.
Thus either $\Spin(9)\cdot\gS^1\subset\gH$ or $\Spin(10)\subset\gH$.
If $\gH$ is a strict equal rank enlargements $\Hi$, then it is isomorphic to
$\Sp(1)\cdot \Spin(9)$ or $\Spin(11)$ and
the central element $\iota\in \Spin(9)$  of $\Hi$ would remain central in $\gH$, which
is impossible as these groups are not in $C(\iota)$.

If $\gH=\Spin(9)\cdot\gS^1 $ we can choose another involution $\iota_2\in  Z(\Spin(8))\cap\gH$ with $\gH^{\iota_2}=
\Spin(8)\cdot\gS^1$. But then  $\Fis=
\Spin(10)\cdot\gS^1/\Spin(8)\cdot\gS^1$ which does not admit positive curvature.

In the second case $\gH=\Spin(10)$,
we can  choose an involution
$\iota_2\in\SU(2)\subset\Spin(4)\subset\Spin(10)$ with $\gH^{\iota_2}=\Spin(4)\Spin(6)$ and
as before $C(\iota_2)\cong \Spin(10)\cdot\gS^1$ or $\SU(6)\cdot\Sp(1)$
but in either case $C(\iota_2)/\gH^{\iota_2}$ can not have positive curvature by our induction hypothesis.

\subsection{$\gK=\gE_7$} We let $\iota$ denote an involution in $\gH$ which is not central in $\gE_7$.
The potential candidates for $C(\iota)$ are $\gE_6\cdot\gS^1$,  $(\Spin(12)\times\Sp(1))/\Delta\Z_2$ or $\SU(8)/\Z_2$.
Notice that $\iota$ must be contained in the center of $C(\iota)$.
The center of $\gE_6\cdot \gS^1$ and $\SU(8)/\Z_2$ only contains one involution and this must be contained in the center of
$\gE_7$ which is $\Z_2$ -- a contradiction.

Thus $C(\iota)=(\Spin(12)\times\Sp(1))/\Delta\Z_2$.
Then $\Hi=\Spin(11)\cdot\Sp(1)=\gH$ or $\Hi=\Spin(12)=\gH$.
But then we can choose another involution $\iota_2\in Z(\Spin(4))\subset \Spin(11)\cap\gH$ with
$\gH^{\iota_2}= \Spin(4)\Spin(7)\Sp(1)$ or $\Spin(4)\Spin(8)$.
As before $C(\iota_2)\cong\Spin(12)\cdot\Sp(1)$ and
the fixed point set does not have positive curvature.

\subsection{$\gK=\gE_8$}
 Here $C(\iota)$ is either $\gE_7\cdot\Sp(1)$ or $\SO'(16)$.
In the former case, we would get $\gH=\gE_7$ and there is
a Riemannnian submersion from $\gE_8/\gE_7$ to the higher rank symmetric space $\gE_8/\gE_7\cdot \Sp(1)$
as the isotropy representation of $\gE_8/\gE_7 \Sp(1)$ remains irreducible when restricted to $\gE_7$ -- a contradiction.

Thus
$\Fi=\SO'(16)/\Hi$ and hence $\Hi=\Spin(15)=\gH$.  We can now choose another involution
$\iota_2\in Z(\Spin(12))\subset \gH$ with $\gH^{\iota_2}=\Spin(12)\Spin(3)$
As before we must have $C(\iota_2)\cong \SO'(16)$ but then $C(\iota_2)/\gH^{\iota_2}$
does not have positive curvature.

\bigskip

\subsection{$\gK=\Sp(2)$ and $\gH= \diag(e^{i\theta},e^{3i\theta})$}\label{subsec:31}
Finally, we discuss the example left out in Section~\ref{subsec:sp2}.\\[2ex]
{\bf Claim.} Without loss of generality $G$ commutes
with $\Ad_a$ where $a=\diag(j,j)$.\\[2ex]
We first want to explain why it is enough to prove the claim.
By the claim it suffices  to consider $\Ad_{\hat{\gH}}$-invariant metrics
with \[\hat{\gH}:= \gH\cup a\gH\cong \mathsf{Pin}(2).\]
Of course any such metric descends via the two fold cover
$\gK/\gH\rightarrow \gK/\hat{\gH}$.
Notice that $\Ad_{a|\Lp}$ has a negative determinant
and thus right multiplication with $a$ induces an
orientation reversing isometry of $\Sp(2)/\gH$ and
$\Sp(2)/\hat{\gH}$ is a non-orientable manifold.
On the other hand, we know that a positively curved odd dimensional manifold is orientable by
 the Synge Lemma.\\[1ex]

It remains to verify the claim. The element $b=\diag(e^{i\psi},e^{i\psi})$ is in the normalizer of $\gH$.
Therefore the
isometry type  does not change if we replace
$G$ by $G_b=\Ad_bG\Ad_{b^{-1}}$ to define an induced metric on $\Sp(2)/\gH$.
We plan to show that for a suitable choice of $b$ the endomorphisms $G_b$ and $\Ad_a$ commute.

Consider the isotropy decomposition
$\Lp=\Lp_0\oplus\Lp_2\oplus\Lp_4\oplus\Lp_6$  preserved by $G$, where $\gH$
acts trivially on the 1-dimensional space $\Lp_0$ and with weight $n$ on $\Lp_n$, $n=2,4,6$.
These subspaces can be described explicitly as follows: $\Lp_0=\R\diag(-3i,i)$,
$\Lp_6=\C\cdot \diag(0,j)$,
$$
\Lp_2=\left\{\left(
         \begin{array}{cc}
           w & -\bar z \\
           z & 0 \\
         \end{array}
       \right) \biggm| w\in \C\cdot j,\ z\in\C
       \right\}\mbox{ and }\quad
\Lp_4=\left\{\left(
         \begin{array}{cc}
           0 & w \\
           w & 0 \\
         \end{array}
       \right) \biggm| w\in \C\cdot j
       \right\}.
$$
In particular, $\Lp_2$ is 4-dimensional, whereas $\Lp_4$ and $\Lp_6$ are 2-dimensional.
Notice that each of these spaces is also $\Ad_a$ invariant.
Since $G_b$ restricts to a multiple of the identity on $\Lp_i$ for $i=0,4,6$, it remains
to show that $G_{b|\Lp_2}$ commutes with $\Ad_{a|\Lp_2}$ for a suitable choice of $b$.

The action of $\gH$ induces a natural complex structure on $\Lp_2$
and we can view $G_{|\Lp_2}$ as hermitian endomorphism with respect to this complex
structure and the scalar product $Q$. The element $\Ad_{a|\Lp_2}$ corresponds to complex conjugation if
we identify the real vector subspace
\[W=\left\{\left(
         \begin{array}{cc}
           \beta j & -\delta\\
           \delta & 0 \\
         \end{array}
       \right) \biggm| \delta,\beta\in \R
       \right\}.\]
 of $\Lp_2$ with $\R^2$. Hence if we consider the hermitian $2\times 2$
 matrix representing $G_{|\Lp_2}$ with respect to an orthonormal basis of $W$, then
 $G$ commutes with $\Ad_a$ if this matrix is real.
 If we replace $G$ by $G_{b}$ the corresponding
 $2\times 2$ matrix changes by conjugating it with
 $\diag(e^{i2\psi},1)$  and clearly we can turn the given
 hermitian matrix into a symmetric real matrix for a suitable choice of $\psi$.

\section*{Final Remarks}
An analysis of the proof shows:
If a simply connected compact homogeneous space
$\gK/\gH$ satisfies the conclusion of the Berger Lemma, but does not admit an invariant metric
of positive sectional curvature
then either one can find commuting eigenvectors
or $\gK/\gH$ is given by $\Sp(2)/\diag(z^3,z)$, $\Sp(2)/\diag(z,z)$
or $\Sp(3)/\{\diag(z,z,g)\mid z\in \gS^1,g\in \gS^3\}$.
The first space is ruled out in section~\ref{subsec:31}. It is also
 the most difficult case in \cite{BB} where one finds another proof that it does not admit positive curvature,
 by exhibiting two commuting vectors (not necessarily eigenvectors) with 0 curvature.
 The third space contains the second one as a totally geodesic submanifold.
 It was pointed out by M. Xu and J. Wolf, that
 B\'erard Bergery did not consider the most general class of metrics on the second space
 when he tried to rule out positive curvature. They also show that one can find metrics on it
 where all planes spanned by commuting vectors have positive curvature.
 We recall that we rule out this potential example  in section~\ref{subsec:sp2} using part b) of Lemma~\ref{ev}.

\providecommand{\bysame}{\leavevmode\hbox
to3em{\hrulefill}\thinspace}


\begin{thebibliography}{99999}


\bibitem[AW]{AW} S.~Aloff and N.~Wallach, \emph{ An infinite family of
7--manifolds admitting positively curved Riemannian structures},
Bull. Amer. Math. Soc. {\bf 81}(1975), 93--97.

\bibitem[BB]{BB} L.~B\'erard Bergery, {\em  Les vari\'et\'es
riemanniennes homog\`enes simplement connexes de dimension impaire
\`a courbure strictement positive}, J.\ Math.\ pure et appl.\ {\bf
55}\,(1976), 47--68.

\bibitem[Be]{Be} M.~Berger, {\em Les varietes riemanniennes homogenes
normales simplement connexes a courbure strictment positive}, Ann.
Scuola Norm. Sup. Pisa {\bf 15} (1961), 191-240.


\bibitem[FGT]{FGT}
F.~Fang, K.~Grove,  and G.~Thorbergsson, \emph{ Tits Geometry and Positive Curvature}, arXiv:1205.6222.



\bibitem[GWZ]{GWZ}
K.~Grove, B.~Wilking and W.~Ziller, \emph{Positively curved
cohomogeneity one manifolds and 3-Sasakian geometry}, J. Diff. Geom. \textbf{78} (2008), 33--111.


\bibitem[GZ]{GZ} K.Grove and W.Ziller, \emph{Cohomogeneity one manifolds with
positive Ricci curvature},
      Inv. Math.  149 (2002), 619-646.

\bibitem[GS]{GS} K.~Grove and K.~Shankar \emph{Rank 2 fundamental groups of positively curved manifolds}
J. Geom. Ana. 10 (2000),679-682.

\bibitem[P\"{u}]{Pu}
T.~P\"{u}ttmann, \emph{Optimal pinching constants of odd dimensional
homogeneous spaces}, Inv. Math. \textbf{138}, (1999), 631–-684.


\bibitem[Sh1]{Sh1} K.\ Shankar,
\emph{On the fundamental groups of positively curved manifolds}, J.
Diff. Geom. \textbf{49} (1998), 179--182.

\bibitem[Sh2]{Sh2} K.\ Shankar,
\emph{Isometry groups of homogeneous, positively curved manifolds}, Diff. Geometry and its Appl., \textbf{14} (2001), 57--78.

\bibitem[Wa]{Wa} N.~Wallach, {\em Compact homogeneous Riemannian
manifolds with strictly positive curvature}, Ann. of Math. {\bf 96}
(1972), 277-295.


     \bibitem[Wi1]{Wi1} B.~Wilking, {\em The normal homogeneous space
$\SU(3)\times \SO(3)/\U(2)$ has positive sectional curvature,} Proc.
     of  Amer. Math. Soc.
     {\bf 127} (1999), 1191-1994.


\bibitem[Wi2]{Wi2} B.~Wilking,
       {\em Positively curved manifolds with symmetry,} Ann. of
       Math. \textbf{163} (2006), 607-668.

\bibitem[Wo]{Wo} J.~Wolf, {\em Spaces of Constant Curvature},
 Fifth edition, Publish or Perish Inc., (1984),  MR0217740, Zbl
 0162.53304.

 \bibitem[WX]{WX} J.~Wolf and M. Xu, {\em $\Sp(2)/\U(1)$ and a positive curvature problem},
arxiv.org/abs/1502.02755

 \bibitem[Zi]{Zi} W.Ziller,
 \emph{Examples of Riemannian manifolds with nonnegative sectional curvature},
  in: Metric and Comparison Geometry,  Surv.  Diff. Geom. 11,
  ed. K.Grove and J.Cheeger, (2007), 63--102.

\end{thebibliography}
\end{document}